\documentclass[reqno]{amsart}

\usepackage{amsmath,amssymb,latexsym, amsfonts, amscd, amsthm}
\usepackage[titletoc]{appendix}

\usepackage{hyperref}
\hypersetup{colorlinks=false}

\newtheorem{thm}{Theorem}[section]
\newtheorem{prop}[thm]{Proposition}
\newtheorem{lem}[thm]{Lemma}

\theoremstyle{definition}
\newtheorem{definition}[thm]{Definition}

\usepackage[foot]{amsaddr}

\title[On supersingular representations of $GL_2(D)$ over a $p$-adic field]{On supersingular representations of $GL_2(D)$ with a division algebra $D$ over a $p$-adic field}
\author[Wijerathne Mudiyanselage Menake Wijerathne]{Wijerathne Mudiyanselage Menake Wijerathne}

\keywords{Supersingular representation, Mod-$p$ local Langlands, Modular representation}
\subjclass[]{Primary \textbf{11F70}; Secondary 11F85}

\begin{document}
\maketitle  
\begin{abstract}
We investigate the mod-$p$ supersingular  representations of $GL_2(D)$, where $D$ is a division algebra over a $p$-adic field with characteristic 0, by computing a basis for the vector space of the pro-$p$ Iwahori subgroup invariants of a certain quotient of a compact induction. This work generalizes the results of Hendel and Schein.
\end{abstract}
\setcounter{tocdepth}{1}
\section{Introduction} 
\subsection{Motivation}
Let $F$ be a finite extension of $\mathbb{Q}_p$ with ring of integers $\mathcal{O}_F$ and residue field $\mathbb{F}_q$ where $q=p^f$ and $f$ be the inertia degree of $F/\mathbb{Q}_p$. 
In 1994, Barthel and Livne initiated the study of smooth irreducible mod-$p$ representations of $GL_2(F)$ in \cite{MR1290194}. This study played a major role in establishing the mod-$p$ Langland correspondence in \cite{MR2931521} and \cite{MR2827792}. Let $F$ be a finite extension of $\mathbb{Q}_p$ with ring of integers $\mathcal{O}_F$ and residue field $\mathbb{F}_q$ where $q=p^f$ and $f$ be the inertia degree of $F/\mathbb{Q}_p$. They showed that any irreducible representation of $GL_2(F)$ originates from the quotient of \begin{equation*}
    \text{ind}_{KZ}^G\sigma/(T-\lambda)\text{ind}_{KZ}^G\sigma,\end{equation*} where $\lambda \in \overline{\mathbb{F}}_p$, $\text{ind}_{KZ}^G\sigma$ is the compact induction of an irreducible smooth representation $\sigma$ of $KZ$ (with $K=GL_2(\mathcal{O}_F)$ and $Z$ the center of $GL_2(F)$), and $T$ is the standard spherical Hecke operator (defined in Proposition 8 in \cite{MR1290194}). Their classification included a new kind of representations, called supersingular representations, which were attained when $\lambda = 0$. Breuil proved that these are irreducible for the case $F=\mathbb{Q}_p$ by explicitly computing the $I(1)$-invariant subspace of $\text{ind}_{KZ}^G\sigma/(T)$, where $I(1)$ is the pro-$p$-Iwahori subgroup of $GL_2(\mathbb{Z}_p)$ (see Theorem 3.2.4 in \cite{MR2018825}). Schein (see \cite{MR2833554} and \cite{MR3195399}) and Hendel (see \cite{MR3873949}) generalized these methods and obtained an infinite dimensional basis of $I(1)$-invariant subspace for the quotient $\text{ind}_{KZ}^G\sigma/(T)$ in the case of $F\neq \mathbb{Q}_p$.
\subsection{Method and Results}
In this paper, we primarily employ the methods given in Hendel's paper (see \cite{MR3873949}) to obtain an infinite dimensional basis of the $I(1)$-invariant subspace of the quotient $\text{ind}_{KZZ^{\prime}}^G\sigma/(T)$ of $GL_2(D)$. Here, $D$ is a finite dimensional central $F$-division algebra of dimension $w^2$ with ring of integers $\mathcal{O}_D$ and residue field $\mathbb{F}_{q^w}$. $K=GL_2(\mathcal{O}_D)$, Z is the center, and $Z^{\prime}$ is the group generated by $\begin{pmatrix} 
\varpi_D & 0  \\ 
0 & \varpi_D
 \end{pmatrix}$. $T$ is a function in the Hecke algebra that is supported on the double coset \begin{equation*}
     KZZ^{\prime}\begin{pmatrix} 1 & 0  \\ 0& \varpi_D\end{pmatrix}KZZ^{\prime} \end{equation*} (see Definition \ref{Hecke} for a full description). 

Let us denote $[g, v]$ as an element of $\textnormal{ind}_{KZZ^{\prime}}^G\sigma$, where $g\in G$ and $v\in V_{\sigma}$, as first introduced by Breuil in \cite{MR2018825}. For a full description of this notation, see section \ref{weight}. Additionally, let $I_n=\{\sum_{i=0}^{n-1}\varpi_D^i[\mu_i]\hspace{3pt}:\hspace{3pt}\mu_i\in \mathbb{F}_{q^w}\}$ be a set of explicit coset representatives of $\mathcal{O}_D/(\varpi_D^n)$ with $I_0=\{0\}$. 

As stated in Lemma \ref{1-1}, there is a one-to-one correspondence between the irreducible representations of the maximal compact subgroup $K = GL_2(\mathcal{O}_D)$ and the irreducible representations, known as weights, of $G = GL_2(k_D)$.

Now, we define the following elements in $\text{ind}_{KZZ^{\prime}}^G\sigma$ for $n\geq1$: 
\begin{align*}
    &s_n^k=\sum\limits_{\mu\in I_n}[g_{n, \mu}^0, \mu^{k}_{n-1}\bigotimes\limits_{j=0}^{wf-1}x_j^{r_j}] \hspace{5pt}\text{where}\hspace{5pt} 0\leq k \leq q^w-1,\\
    &t_n^s=\sum\limits_{\mu\in I_n}[g_{n, \mu}^0,  \bigotimes\limits_{s\neq j=0}^{wf-1}x_j^{r_j}\otimes x_s^{r_s-1}y_s]\hspace{5pt}\text{where}\hspace{5pt} 0\leq s \leq wf-1, 
\end{align*}
where $\mu =\sum\limits_{i=0}^{n-1}\varpi_D^i[\mu_i]$.
Propositions \ref{prop9} and \ref{invariant} provide comprehensive information regarding the notations mentioned previously. Consequently, the main theorem of this paper is:
\begin{thm}\label{mainmain}$($Theorem \ref{main}$)$
Let $D$ be a finite dimensional central $F$-division algebra of dimension $w^2$ over $F$, and let $e$ and $f$ be the ramification degree and inertia degree of $F$ over $\mathbb{Q}_p$, respectively. Assume $2<r_j< p-3$ for $0\leq j \leq wf-1$, and let us set,
\begin{align*}
    &S_m^l=\{s_n^{p^l(r_l+1)}\}_{n\geq m}\cup \{\beta s_n^{p^l(r_l+1)}\}_{n\geq m},\\
    &T_m^l=\{t_n^l\}_{n\geq m}\cup \{\beta t_n^l\}_{n\geq m},\\
   &S_m=\cup_{l=0}^{wf-1}S_m^l,\\
    &T_m=\cup_{l=0}^{wf-1}T_m^l.\\
\end{align*}
Then, an $I$-eigenbasis for the space $(\textnormal{ind}_{KZZ^{\prime}}^G\sigma/(T))^{I(1)}$ of $I(1)$-invariant as an $\overline{\mathbb{F}}_p$-vector space is given by
\begin{align*}
   S_1\cup \{[\textnormal{Id}, \bigotimes\limits_{j=0}^{wf-1} x_j^{r_j}], [\alpha,\bigotimes\limits_{j=0}^{wf-1} y_j^{r_j} ]\} &\hspace{10pt} \text{when} \hspace{5pt}e=1,wf>1,\\
   S_1\cup \{[\textnormal{Id},\bigotimes\limits_{j=0}^{wf-1} x_j^{r_j}], [\alpha,\bigotimes\limits_{j=0}^{wf-1} y_j^{r_j} ]\}\cup T_1 &\hspace{10pt} \text{when}\hspace{5pt} e>1,wf>1.\\
\end{align*}

\end{thm}

\subsection{Organization.}
This paper is organized as follows. In Section 2, we discuss some preliminary results on mod-$p$ representations of $GL_2(D)$ and prove in Lemma \ref{1-1} a bijection between every irreducible smooth representations of the maximal compact subgroup $K=GL_2(\mathcal{O}_D)$ and the irreducible representations of $GL_2(K_D)$. In Section 3, we examine the properties of $I(1)$-invariants and Section 4 is devoted to prove our main theorem by obtaining a basis for the space $(\textnormal{ind}_{KZZ^{\prime}}^G\sigma/(T))^{I(1)}$. Several of these results are obtained by extending Hendel's work in \cite{MR3873949} to the case of $GL_2(D)$.
\subsection*{Acknowledgement} I would like to extend my heartfelt gratitude to my PhD advisor, Kwangho Choiy, for his unwavering support, guidance and encouragement throughout this research project. He has provided invaluable advice and feedback throughout the entire process and I could not have completed this project without his constant help and guidance. 

\setcounter{tocdepth}{1}

\section{Notations and preliminaries} \label{}
Let $D$ be a finite dimensional central $F$-division algebra of dimension $w^2$ over $F$. Let $\mathcal{O}_D$ denote its ring of integers with uniformizer $\varpi_D$. Let $e$ and $f$ be the ramification degree and inertia degree of $F/\mathbb{Q}_p$, respectively. The residue field $k_D = \mathcal{O}_D/(\varpi_D)$ is isomorphic to $\mathbb{F}_{q^w}$, where $q=p^f$. Consider the multiplicative group $G=GL_2(D)$ consisting of all $2\times 2$ invertible matrices whose entries are from $D$. Let $K$ denote the maximal compact open subgroup of $G$.  

\begin{lem}\label{1-1}
There is a bijection between the irreducible smooth representations of $K=GL_2(\mathcal{O}_D)$ and the irreducible representations of $GL_2(k_D)$.
\end{lem} 

\begin{proof}
We provide a proof here for the convenience of the reader, which follows the ideas given in Section 2.3 of \cite{ly}. Since $K(1)=1+\varpi_D M_2(\mathcal{O}_D)$ is an open, normal subgroup of $K$, and $K(1)$ is pro-$p$, it follows that $V^{K(1)}\neq 0$ for any smooth irreducible representation $V$ of $K$. Moreover, since $K(1)$ is normal in $K$, we have that $V^{K(1)}$ is a $K$-stable subspace of $V$. As a result, we can conclude that $V$ is a representation of $K/K(1)\cong GL_2(k_D)$, and $V=V^{K(1)}$. Conversely, for a smooth irreducible representation $\sigma$ of $K/K(1)$, we define a map $\varphi$ such that
\begin{equation*}
\varphi :K \longrightarrow K/K(1).
\end{equation*}
We define an inflation map $\widetilde{\sigma}$ that takes elements of $K/K(1)$ to $K$ such that
\begin{equation*}
\widetilde{\sigma}(k) = \sigma(\varphi(k))
\end{equation*}which yields the desired irreducible representation.
\end{proof}
\subsection{Weights}
From this point forward, we shall refer to the irreducible smooth mod-$p$ representations of $GL_2(k_D)$ as weights. For a given vector $\vec{r}=(r_0,\cdots,r_{wf-1})\in\{0,1,....,p-1\}^{wf}$ where $wf$ is the degree of the finite extension of $k_D$ over $\mathbb{F}_p$, a model for the weight is denoted by $(V_{\sigma})$ and is given by the space of homogeneous polynomials $Sym^{\vec{r}}\overline{\mathbb{F}}_p=\bigotimes_{j=0}^{wf-1}Sym^{r_j}\overline{\mathbb{F}}_p$ with the basis $\{\bigotimes_{j=0}^{wf-1} x_j^{r_j-i_j}y_j^{i_j}\}_{\vec{i}=0}^{\vec{r}}$. The action of $\sigma$ on $GL_2(k_D)$ is then given by,
\begin{equation} \label{weight}
    \sigma\begin{pmatrix} 
a& b  \\ 
c& d
 \end{pmatrix}(\bigotimes_{j=0}^{wf-1} x_j^{r_j-i_j}y_j^{i_j})=\bigotimes\limits_{j=0}^{wf-1}(a^{p^j}x_j+c^{p^j}y_j)^{r_j-i_j}(b^{p^j}x_j+d^{p^j}y_j)^{i_j}.
\end{equation}

By abuse of notation, we let $\sigma$ denote an irreducible smooth representation of $K$ on an $\overline{\mathbb{F}}_p$-vector space corresponding to the vector $V_{\sigma}$. We inflate $\sigma$ to a representation of $KZZ^{\prime}$ by making $Z$ and $Z^{\prime}$ act trivially. We define the compact induction of $\sigma$ from $KZZ^{\prime}$ to $G$, as follows:
\begin{equation*}
    \textnormal{ind}_{KZZ^{\prime}}^G\sigma:=\{f:G\longrightarrow V_{\sigma}|f(hg)=\sigma(hg)=\sigma(h)(f(g)), \hspace{5pt}\forall g\in G, h\in KZZ^{\prime}\}
\end{equation*}
with the action of $g\in G$ on $f$ is given by $(gf)(g^\prime)=f(g^\prime g)$ and $f(kg)=\sigma(k)(f(g))$ when $k\in KZZ^{\prime}$.
For $g\in G, v\in V_\sigma$, we denote $[g, v]$ to be an element of $\textnormal{ind}_{KZZ^{\prime}}^G\sigma$ defined as follows,
\begin{align*}
[g,v](g^\prime)&=\sigma(g^\prime g).v \hspace{5pt}if\hspace{5pt} g^\prime\in KZZ^{\prime}g^{-1},\\
[g,v](g^\prime)&=0\hspace{5pt} if\hspace{5pt} g^\prime\notin KZZ^{\prime}g^{-1}.
\end{align*}
Thus, the action of $g\in G$ on $[g, v]$ is given by $g^\prime([g, v])=[{g^\prime}g, v]$ and moreover $[gk, v]=[g, \sigma(k)v]$ for $k\in K$. Also, any elements $f\in \textnormal{ind}_{KZZ^{\prime}}^G\sigma$ can be written as $f=\sum\limits_{i\in I}[g_i, v_i]$, where $g_i\in G, v_i\in V_\sigma$ and $I$ is a finite set $($see, section 2.3 in \cite{MR2018825}$)$.
\subsection{Hecke Algebra}
\begin{definition}\label{Hecke}
The Hecke algebra of the weight $V_{\sigma}$,
\begin{equation*}
    \mathcal{H}(V_{\sigma}):=\text{End}_G(\text{ind}_{KZZ^{\prime}}^G\sigma).
\end{equation*}
We have that $\mathcal{H}(V_{\sigma})$ is isomorphic to the $\overline{\mathbb{F}}_p$-vector space of functions $\varphi:G\longrightarrow End(V_\sigma)$ such that \begin{equation}\label{phi}
\varphi(k_1gk_2)=k_1\cdot\varphi(g)\cdot k_2
\end{equation}
for $k_1, k_2\in KZZ^{\prime}$ and $g\in G$ $($see, section $2.2$ in \cite{MR1290194}$)$. 

Let $\varphi$ be such a function, 
$ g\in G$, $v\in V_\sigma$ and $T$ is a function in Hecke algebra supported on the double coset $KZZ^{\prime}\begin{pmatrix} 
1 & 0  \\ 
0& \varpi_D
 \end{pmatrix}KZZ^{\prime}$, such that $T\begin{pmatrix} 
1 & 0  \\ 
0& \varpi_D
 \end{pmatrix}\in \text{End}_G(V_{\sigma})$ is a linear projection$($see, Proposition 45 in \cite{MR3495745}$)$. 
\end{definition}
Then we have the formula $($see, equation 8 in \cite{MR1290194}$)$,
\begin{equation} \label{hecke}
    T([g,v])=\sum\limits_{g^\prime KZZ^{\prime}\in G/KZZ^{\prime}}[gg^\prime, \varphi(g^{\prime -1})(v)].
\end{equation}
\subsection{Cartan Decomposition.}
We let,
\begin{align*}
        I&= \{\begin{pmatrix} 
a& b  \\ 
\varpi_D c& d
 \end{pmatrix} \mid a\in\mathcal{O}_D, \hspace{5pt}d\in\mathcal{O}_D,\hspace{5pt} b\in\mathcal{O}_D, \hspace{5pt}c\in\mathcal{O}_D\}, \\
 I(1) &= \{\begin{pmatrix} 
a& b  \\ 
\varpi_D c& d
 \end{pmatrix} \mid a\equiv d \equiv 1 \hspace{5pt}\text{mod}(\varpi_D), \hspace{5pt}b\in\mathcal{O}_D, \hspace{5pt}c\in\mathcal{O}_D\},
\end{align*}
where $I$ is the Iwahori subgroup and $I(1)$ is the pro-$p$-Sylow subgroup of $I$. Also we set,
\begin{equation}\label{alpha}
\alpha=\begin{pmatrix} 
1& 0  \\ 
0 & \varpi_D
 \end{pmatrix},\hspace{5pt} \beta=\begin{pmatrix} 
0& 1  \\ 
\varpi_D & 0
 \end{pmatrix} \hspace{5pt}\text{and}\hspace{5pt} w=\begin{pmatrix} 
0& 1  \\ 
1 & 0
 \end{pmatrix}.
 \end{equation} 
We have the Cartan decomposition of $G$ as follows:
\begin{equation}\label{cartan}
G=\bigsqcup_{n\in \mathbb{N}} KZZ^{\prime}\alpha^{-n}KZZ^{\prime} =\bigsqcup_{n\in \mathbb{N}} \left(IZZ^{\prime}\alpha^{-n}KZZ^{\prime}\sqcup IZZ^{\prime}\beta\alpha^{-n}KZZ^{\prime}\right), 
\end{equation}
where $Z$ is the center of $G$ and $Z^{\prime}$ is the group generated by $\begin{pmatrix} 
\varpi_D& 0  \\ 
0 & \varpi_D
 \end{pmatrix}$, with $\beta$ normalizing $I(1)$.
Note that $KZZ^{\prime}\alpha^{-n}I=KZZ^{\prime}\alpha^{-n}I(1)$, $IZZ^{\prime}\alpha^{-n}I=IZZ^{\prime}\alpha^{-n}I(1)$ and $IZZ^{\prime}\beta\alpha^{-n}I=IZZ^{\prime}\beta\alpha^{-n}I(1)$(see, \cite{MR2018825}, \cite{MR3873949} and \cite{MR1769724}).

Let $\mathcal{A}$ be the Bruhat-Tits tree of $GL_2(k_D)$, and its vertices are in equivariant bijection with the coset $G/KZZ^{\prime}$, as stated in section 2.2 of \cite{MR2018825}. We can describe a model for $\textnormal{ind}_{KZZ^{\prime}}^G\sigma$ as follows. Let $I_0 = \{0\}$ and for $n\in\mathbb{N}$, 
\begin{equation*}
    I_n=\{[\mu_0]+\varpi_D[\mu_1]+\varpi_D^2[\mu_2]+\dots+\varpi_D^{n-1}[\mu_{n-1}]: \mu_{i}\in k_D\}\subset \mathcal{O}_D
\end{equation*}
where $[\cdot]$ is the multiplicative representative in $\mathcal{O}_D$. If $0\leq m\leq n$, let $[\cdot]_m: I_n\longrightarrow I_m$ be the truncation map define by
\begin{equation*}
    \sum_{i=0}^{n-1}[\mu_i]\varpi_D^i\longrightarrow \sum_{i=0}^{m-1}[\mu_i]\varpi_D^i.
\end{equation*}
\begin{lem}\label{witt}
Let, $\mu, \lambda\in \mathbb{F}_{q^w}$, then $[\mu]+[\lambda]\equiv [\mu+\lambda]+\varpi_D^e[P_0(\mu, \lambda)] \hspace{5pt}mod\hspace{5pt} \varpi_D^{e+1}$, where
\begin{equation}\label{P_0}
    P_0(\mu, \lambda)=\frac{\mu^{q^{we}}+\lambda^{q^{we}}-(\mu+\lambda)^{q^{we}}}{\varpi_D^e}.
\end{equation}
\begin{proof}
This can be obtained by following Lemma 2.2 in \cite{MR2833554}. 
\end{proof}
\end{lem}
For $n\in \mathbb{N}$ and $\mu\in I_n, $ then we define:
\begin{align*}
   & g_{n,\mu}^0=\begin{pmatrix} 
\varpi_D^n & \mu  \\ 
0& 1
 \end{pmatrix}, \hspace{4pt}g_{n,\mu}^1=\begin{pmatrix} 
1 & 0  \\ 
\varpi_D\mu& \varpi_D^{n+1}
 \end{pmatrix},\\
  &g_{0,0}^0=
\begin{pmatrix} 1 & 0  \\ 0& 1\end{pmatrix} 
= \textnormal{Id},\hspace{4pt} g_{0,0}^1=
\begin{pmatrix} 1 & 0  \\ 0 & \varpi_D\end{pmatrix}
=\alpha.
\end{align*}
Also, we have the relation $$\beta g_{n, \mu}^0=g_{n, \mu}^1w.$$ \\With the above notations we write, $$IZZ^{\prime}\alpha^{-n}KZZ^{\prime}=\bigsqcup_{\mu\in I_n}g_{n, \mu}^0KZZ^{\prime} \hspace{4pt}\text{and}\hspace{4pt} IZZ^{\prime}\beta\alpha^{-n}KZZ^{\prime}=\bigsqcup_{\mu\in I_n}g_{n, \mu}^1KZZ^{\prime}.$$ Thus, $g_{n, \mu}^0$ and $g_{n, \mu}^1$ form a system of representatives of $G/KZZ^{\prime}.$ Therefore, we rewrite the Cartan decomposition \eqref{cartan} as follows, 
\begin{equation*}
    G=(\bigsqcup_{\mu, n}g_{n,\mu}^0KZZ^{\prime})\bigsqcup(\bigsqcup_{\mu, n}g_{n,\mu}^1KZZ^{\prime}).
\end{equation*}
Also, for $n\in \mathbb{N}$, we set,

\begin{align*}
X_n^0=IZZ^{\prime}\alpha^{-n}KZZ^{\prime} \hspace{20pt} & \hspace{20pt}X_n^1=IZZ^{\prime}\beta\alpha^{-n}KZZ^{\prime}\\
B_n^0=\bigsqcup_{m<n}X_m^0 \hspace{20pt}&\hspace{20pt}B_n^1=\bigsqcup_{m<n}X_m^1\\
X_n=X_n^0\bigsqcup X_n^1\hspace{20pt} & \hspace{20pt} B_n=B_n^0\bigsqcup B_n^1.
\end{align*}
\subsection{Three Basic Results}
We will often be employing Lucas' classical theorem in modular combinatorics \cite{MR3873949}, which provides a condition for a binomial coefficient $\binom{m}{n}$ to be congruent to zero modulo $p$.
 \begin{thm}{(Lucas's theorem)}\label{lucus}
The following congruence relation hold for $m, n \in \mathbb{Z}_{>0}$ and a prime $p$,
\begin{equation*}
    \binom{m}{n}=\prod_{i=0}^k\binom{m_i}{n_i}\hspace{5pt} (\text{mod} \hspace{5pt}p),
\end{equation*}
where, $m=\sum_{i=0}^k m_ip^i$ and $n=\sum_{i=0}^kn_ip^i$ are the $p$-adic expansion of $m$ and $n$ respectively. By the convention $\binom{m}{n}=0$ if and only if $m_i<n_i$ for some $i$.
\end{thm}
\begin{proof}
See \cite{MR2557851}.
\end{proof}
\begin{lem}\label{poly}
For $n\in \mathbb{N}$ and $\mu \in I_n$, let $\varphi: I_n \longrightarrow \overline{\mathbb{F}}_p$ be a set-theoretic map for $n\geq 1$. Then there exists a unique polynomial $c(z_0, z_1, \cdots, z_{n-1}) \in \overline{\mathbb{F}}_p[z_0, z_1, \cdots, z_{n-1}]$, such that each variable has degree at most $q^w-1$ and $\varphi(\mu) = c(\mu_0, \mu_1, \cdots, \mu_{n-1})$ for all $\mu \in I_n$.
\end{lem}
\begin{proof}
We use Lemma 3.1.6 from \cite{MR2018825} and Lemma 2.1 from \cite{MR2833554} to prove our claim via induction. For $n = 1$, we obtain a set map $\varphi: \mathbb{F}_{q^w}\longrightarrow \overline{\mathbb{F}}_p$ such that the matrix

\begin{equation}\label{invert}
    \begin{pmatrix} 
1& 0&\cdots & \cdots& 0  \\ 
1& 1&\cdots & \cdots& 1\\
1& 2&2^2 & \cdots& 2^{q^w-1}\\
\vdots& \vdots&\vdots & \vdots& \vdots\\
1& (q^w-1)&(q^w-1)^2 & \cdots& (q^w-1)^{q^w-1}\\ 
 \end{pmatrix}
\end{equation}
is an invertible modulo $q^w-1$. Thus, there exist a unique $(c_0, c_1, \cdots, c_{q^w-1})\in \overline{\mathbb{F}}_p^{q^w-1} $ such that 
\begin{equation*}
    \sum_{i=0}^{q^w-1}c_i\mu_0^i=\varphi(\mu_0).
\end{equation*}
Assume that the claim is true for $n-1$, $n\geq 2$ and suppose $\mu\in I_{n-1}$. Then there exist unique $(c_0(\mu), c_1(\mu),\cdots, c_{q^w-1}(\mu))\in\overline{\mathbb{F}}_p^{q^w-1}$ such that
\begin{equation*}
    \sum_{i=1}^{q^w-1}c_i(\mu)\lambda^i=\varphi(\mu+[\lambda]\varpi_D^{n-1}),
\end{equation*}
where $c_i(\mu)$ is a unique polynomial of $\mu_0,\mu_1,\cdots, \mu_{n-2}$ for all $i$. The proof is complete.
\end{proof}
\begin{lem}\label{modp} We have, 
\begin{align*}
    \sum_{\lambda\in\mathbb{F}_{q^w}}\lambda^j\hspace{5pt} =\hspace{5pt} 0 \hspace{5pt}&if\hspace{5pt} j\hspace{4pt}\neq \hspace{4pt}q^w-1,\\
    \sum_{\lambda\in\mathbb{F}_{q^w}}\lambda^j\hspace{5pt} =\hspace{5pt} -1 \hspace{5pt}&if\hspace{5pt} j\hspace{4pt}= \hspace{4pt}q^w-1
\end{align*}
where $j\in\{0, 1, 2, \cdots, q^{w}-1\}$.
\end{lem}
\begin{proof}
This can be obtain by a direct computation.
\end{proof}

\section{Properties of I(1)-invariants} \label{}
In this section, we present techniques and proofs for determining $(\textnormal{ind}_{KZZ^{\prime}}^G\sigma/T)^{I(1)}$. We denote $r=\sum_{j=0}^{wf-1}p^jr_j$, where $0 \leq r_j \leq p-1$, for $0\leq j \leq wf-1$, and $\vec{r}$ is the vector $(r_0,\cdots,r_{wf-1})$.

First, let us see the following decomposition of any element in $I(1)$ for $a, b, c, d\in \mathcal{O}_D$,
\begin{equation*}
  \begin{pmatrix} 
1+\varpi_D a& b  \\ 
\varpi_Dc& 1+\varpi_Dd
 \end{pmatrix}  =\begin{pmatrix} 
1& b(1+\varpi_Dd)^{-1}  \\ 
0& 1
 \end{pmatrix}\begin{pmatrix} 
1& 0  \\ 
\varpi_D cz^{-1}& 1
 \end{pmatrix}\begin{pmatrix} 
z& 0  \\ 
0& 1+\varpi_Dd
 \end{pmatrix}
\end{equation*}
where $z=1+\varpi_Da-b(1+\varpi_Dd)^{-1}\varpi_Dc$. Therefore, a given vector is $I(1)$ invariant if it's invariant under the following three types of matrices for $a, b, c, d\in \mathcal{O}_D$,
\begin{equation*}
    \begin{pmatrix} 
1& b  \\ 
0& 1
 \end{pmatrix},\hspace{5pt}\begin{pmatrix} 
1& 0  \\ 
\varpi_Dc& 1
 \end{pmatrix},\hspace{5pt} \begin{pmatrix} 
1+\varpi_Da& 0  \\ 
0& 1
 \end{pmatrix}.
\end{equation*}
\begin{lem}\label{lem1}
Let $c_{\vec{i}}\in \overline{\mathbb{F}}_p$ and $v=\sum\limits_{\vec{i}=0}^{\vec{r} }c_{\vec{i}}
\bigotimes\limits_{j=0}^{wf-1} x_j^{r_j-i_j}y_j^{i_j}.$ Recall that $\alpha
=\begin{pmatrix} 
1& 0  \\ 
0& \varpi_D
\end{pmatrix}$, 
$w
=\begin{pmatrix} 
0& 1  \\ 
1& 0 \end{pmatrix}$
defined in \eqref{alpha} and we set $w_\lambda=\begin{pmatrix} 
0& 1  \\ 
1& -\lambda
 \end{pmatrix}$ for $\lambda\in I_1$. Then we have,
 \begin{align}\label{lem2eq1}
     &\sigma(w)\varphi(\alpha^{-1})\sigma(w_\lambda)(v)= \sum\limits_{\vec{i}=0}^{\vec{r} }c_{\vec{i}}
(-\lambda)^i\bigotimes\limits_{j=0}^{wf-1} x_j^{r_j},\\\label{lem2eq2}
&\varphi(\alpha^{-1})\sigma(w_\lambda w)(v)= \sum\limits_{\vec{i}=0}^{\vec{r} }c_{\vec{i}}
(-\lambda)^{r-i}\bigotimes\limits_{j=0}^{wf-1} y_j^{r_j},
 \end{align}
\end{lem}
where $\varphi$ is the function define in \eqref{phi} and note that in the basis  $\bigotimes_{j=0}^{wf-1}(x_j^{r_j-i_j}y_j^{i_j})_{0\leq i_j \leq r_j}$ of $V_{\sigma}$, we get $\varphi(\alpha^{-1})=\begin{pmatrix} 0& 0  \\ 0& 1\end{pmatrix}$, (see, section 3 in \cite{MR2018825}).
\begin{proof}
We have
\begin{align*}
\sigma(w)\varphi(\alpha^{-1})\sigma(w_\lambda)(v) &= w\cdot\begin{pmatrix} 0& 0  \\ 0& 1\end{pmatrix}
 \cdot w_{\lambda}(v)\nonumber\\
 &=\begin{pmatrix} 1& -\lambda  \\ 0& 0\end{pmatrix}(v). \label{lem2eq3}
 \end{align*}
 Then the action given in \eqref{weight} verifies \eqref{lem2eq1} and we get   
 
\begin{equation*}
\sigma(w)\varphi(\alpha^{-1})\sigma(w_\lambda)(v)=\sum\limits_{\vec{i}=0}^{\vec{r} }c_{\vec{i}}
(-\lambda)^i\bigotimes\limits_{j=0}^{wf-1} x_j^{r_j}.
 \end{equation*}
 Similarly, for \eqref{lem2eq2} we have
 \begin{align*}
 \varphi(\alpha^{-1})\sigma(w_\lambda w)(v)&= \begin{pmatrix} 0& 0  \\ 0& 1\end{pmatrix}\cdot ww_{\lambda}(v)\\
&=\begin{pmatrix} 0& 0  \\ -\lambda & 1\end{pmatrix}(v)\\    
&=\sum\limits_{\vec{i}=0}^{\vec{r}}c_{\vec{i}}(-\lambda)^{r-i}
\bigotimes\limits_{j=0}^{wf-1} y_j^{r_j}.
 \end{align*}
\end{proof}
\subsection{Action of the Hecke operator.}
\begin{prop}\label{T}
Let $T$ be a function in $\textnormal{End}_G(\textnormal{ind}_{KZZ^{\prime}}^G\sigma)$ as define in \eqref{hecke}  and $v=\sum\limits_{\vec{i}=0}^{\vec{r} }c_{\vec{i}}
\bigotimes\limits_{j=0}^{wf-1} x_j^{r_j-i_j}y_j^{i_j}\in V_{\sigma}$, where $c_{\vec{i}}\in \overline{\mathbb{F}}_p$. Then for $n\geq 1$ and $\mu \in I_n$, we have
\begin{align}\label{im(T)0}
    T([g_{n,\mu}^0, v])=\sum\limits_{\lambda\in I_1}[g_{n+1,\mu+\varpi_D^n\lambda}^0,&\sum\limits_{\vec{i}=0}^{\vec{r} }c_{\vec{i}}
(-\lambda)^i\bigotimes\limits_{j=0}^{wf-1} x_j^{r_j}]\\
&+[g_{n-1,[\mu]_{n-1}}^0, c_{\vec{r}}
\bigotimes\limits_{j=0}^{wf-1} ({\mu_{n-1}^{p^j}}x_j+y_j)^{r_j} ],\nonumber\\
T([g_{n,\mu}^1, v])=\sum\limits_{\lambda\in I_1}[g_{n+1,\mu+\varpi_D^n\lambda}^1,&\sum\limits_{\vec{i}=0}^{\vec{r} }c_{\vec{i}}(-\lambda)^{r-i}
\bigotimes\limits_{j=0}^{wf-1}y_j^{r_j}]\nonumber\\
&+[g_{n-1,[\mu]_{n-1}}^1, c_{\vec{0}}
\bigotimes\limits_{j=0}^{wf-1} {(x_j+\mu_{n-1}^{p^j}}y_j)^{r_j} ],\nonumber
\end{align}
and for $n=0$,
\begin{align}
    &T([\textnormal{Id}, v])=\sum\limits_{\lambda\in I_1}[g_{1,\lambda}^0,\sum\limits_{\vec{i}=0}^{\vec{r} }c_{\vec{i}}
(-\lambda)^i\bigotimes\limits_{j=0}^{wf-1} x_j^{r_j} ]+[\alpha, c_{\vec{r}}
\bigotimes\limits_{j=0}^{wf-1} y_j^{r_j} ],\label{id}\\
&T([\alpha, v])=\sum\limits_{\lambda\in I_1}[g_{1,\lambda}^1, \sum\limits_{\vec{i}=0}^{\vec{r} }c_{\vec{i}}(-\lambda)^{r-i}\bigotimes\limits_{j=0}^{wf-1} y_j^{r_j} ]+[\textnormal{Id}, c_{\vec{0}}
\bigotimes\limits_{j=0}^{wf-1} x_j^{r_j} ].\nonumber
\end{align}
\begin{proof}
The result follows from formulas (4) to (9) in \cite{MR2018825} and Lemma \ref{lem1}.
\end{proof}
\end{prop}
\subsection{Pro-$p$ Iwahori invariants of $\textnormal{ind}_{KZZ^{\prime}}^G\sigma$.}
We can now define the elements of $\textnormal{ind}_{KZZ^{\prime}}^G\sigma$ for the weight $\sigma$ as follows:
\begin{align*}
    &s_n^k=\sum\limits_{\mu\in I_n}[g_{n, \mu}^0,  \mu^{k}_{n-1}\bigotimes\limits_{j=0}^{wf-1}x_j^{r_j}], \hspace{5pt}where\hspace{5pt} 0\leq k \leq q^w-1,\\
    &t_n^s=\sum\limits_{\mu\in I_n}[g_{n, \mu}^0,  \bigotimes\limits_{s\neq j=0}^{wf-1}x_j^{r_j}\otimes x_s^{r_s-1}y_s]\hspace{5pt}where\hspace{5pt} 0\leq s \leq wf-1. 
\end{align*}
\begin{prop}\label{prop9}
A basis for the $\overline{\mathbb{F}}_p$-vector space of $I(1)$-invariants of $\textnormal{ind}_{KZZ^{\prime}}^G\sigma$ for weight $\sigma$ is given by the set $\{s_n^0, \beta s_n^0 : n\in \mathbb{N}\}$. Furthermore, we have 
\begin{enumerate}
    \item If $0< r \leq q^w-1$, the $KZZ^{\prime}$-module generated by $s_n^0$ is isomorphic to $\sigma$ and hence irreducible. Also, the $KZZ^{\prime}$-module generated by $\beta s_n^0$ is isomorphic to a reducible $(q^w+1)$-dimensional principal series for all $n\geq 1.$ 
    \item If $r = 0$, the element $s_n^0+\beta s_{n-1}^0$ generates an irreducible one dimensional $KZZ^{\prime}$-module isomorphic to $\sigma$ for all $n \geq 0$.
    \item If $r=q^w-1$, then $s_n^0+\beta s_{n-1}^0$ generates an irreducible one dimensional $KZZ^{\prime}$-module for all $n\geq 1.$
\end{enumerate}
\end{prop}
\begin{proof}
\begin{enumerate}
    \item 
We utilize the methods presented in the proof of Proposition 2.5 in \cite{MR3195399}\\
Following Proposition 14 in \cite{MR1290194} for the case $F=D$, we set
\begin{equation*}
S(G, Sym^{\vec{r}}\overline{\mathbb{F}}_p^2)^{I(1)}=\{\varphi | 
\hspace{4pt}\varphi(kgi)=\sigma(k)\varphi(g)\hspace{2mm} \forall k\in KZZ^{\prime}, g\in G, i\in I(1)\}
\end{equation*}
where $S(G, Sym^{\vec{r}}\overline{\mathbb{F}}_p^2)$ be the space of locally
constant functions on $G$ with values in $Sym^{\vec{r}}\overline{\mathbb{F}}_p^2$ which are
compactly supported modulo $KZZ^{\prime}$ on the left. With a slight change to the arguments given in Proposition 14(2) and Proposition 15 in \cite{MR1290194}, we have  the elements $s_n^0$ and $\beta s_n^0$ as eigenvectors for standard $I$-action.
Note that the element $s_0^0=[\textnormal{Id}, \bigotimes_{j=0}^{wf-1}x_j^{r_j}]$ generates an irreducible $KZZ^{\prime}$-submodule (see, Remark 6 in \cite{MR3495745}) and by Lemma \ref{1-1} it is isomorphic to $\sigma$.
Next, we compute the $KZZ^{\prime}$-submodule generated by $\beta s_0^0=[\alpha,  \bigotimes_{j=0}^{wf-1}y_j^{r_j}]$ using the following coset representatives of $KZZ^{\prime}/I$, 
\begin{equation*}
    \bigg\{ZZ^{\prime}\begin{pmatrix} 
0& 1  \\ 
1& 0   
 \end{pmatrix}I\bigg\}\cup \bigg\{ZZ^{\prime}\begin{pmatrix} 
1& 0  \\ 
\lambda_0& 1   
 \end{pmatrix}I\bigg\}_{\lambda_0\in I_1}
\end{equation*}
(see, Lemma 2.3 in \cite{MR3873949}). Thus, we get
\begin{align*}
   \begin{pmatrix} 
0& 1  \\ 
1& 0   
 \end{pmatrix}[\alpha, \bigotimes_{j=0}^{wf-1}y_j^{r_j}]&= [g_{1, 0}^0, \bigotimes_{j=0}^{wf-1}x_j^{r_j}] \hspace{5pt}\text{and}\\
\begin{pmatrix} 
1& 0  \\ 
\lambda_0& 1\end{pmatrix}[\alpha, \bigotimes_{j=0}^{wf-1}y_j^{r_j}] &= [g_{1, \lambda_0^{-1}}^0, \begin{pmatrix} 
0& -\lambda_0^{-1}  \\ 
\lambda_0& \varpi_D\end{pmatrix}\bigotimes_{j=0}^{wf-1}y_j^{r_j}]\\
&=[g_{1,\lambda_0^{-1}}^0, 
\bigotimes_{j=0}^{wf-1}((-\lambda_0^{-1})^{p^j}x_j+\varpi_D^{p^j}y_j)^{r_j}]\\
&=[g_{1,\lambda_0^{-1}}^0, (-\lambda_0^{-1})^{r}
\bigotimes_{j=0}^{wf-1}x_j^{r_j}]
\end{align*}

where $\lambda_0\in \mathbb{F}_{q^w}^{\times}$ and $\alpha=g_{0, 0}^1$. We observe that there exist $q^w+1$ elements which span the
$KZZ^{\prime}$-submodule generated by $\beta s_0^0$ and all these elements are linearly independent.  

We prove that the $KZZ^{\prime}$-module generated by $\beta s_n^0$ is isomorphic to a reducible $(q^w+1)$-dimensional principal series representation by applying Frobenius reciprocity given in Theorem 5 in \cite{MR3495745}. Since $\beta s_0^0$ is an $I$-eigenvector, we obtain a map $\varphi\in \textnormal{Hom}_I(\chi_{\beta s_0^0}, \textnormal{ind}_{KZZ^{\prime}}^G\sigma_I)$ such that $\varphi(1)=[\alpha, \bigotimes_{j=0}^{wf-1}y_j^{r_j}]$, where $\chi_{\beta s_0^0}$ is the character of $I$ corresponding to $\beta s_0^0$. Using Frobenius reciprocity, we obtain a map of $K$-modules,\\ $\widetilde{\varphi} \in \textnormal{Hom}_K(\textnormal{Ind}_I^K{\chi_{\beta s_0^0}}, \textnormal{ind}_{KZZ^{\prime}}^G\sigma)$ such that $\widetilde{\varphi}([\textnormal{Id}, 1])= [\alpha, \bigotimes_{j=0}^{wf-1}y_j^{r_j}]$. By slightly modifying the argument of Lemma 2.3 in \cite{MR3195399}, we find that $\textnormal{Ind}_I^K{\chi_{\beta s_0^0}}$ is a $q^w+1$-dimensional representation of $K$. Therefore, the module generated by $\beta s_0^0$ is isomorphic to a $q^w+1$-dimensional principal series representation by dimension considerations.

Our claim follows from Proposition \ref{T} since $T$ is an injective map of $G$-modules and $T^n(s_0^0)= s_{n}^0$ and $T^n(\beta s_0^0)= \beta s_{n}^0$.\\
\item For the case $r=0$, using the equation \eqref{id} we have
\begin{equation*}
    T(s_0^0)=\sum\limits_{\lambda\in I_1}[g_{1,\lambda}^0, 1 ]+[\alpha, 1 ]= s_1^0+ \beta s_0^0.
\end{equation*}
Note that $s_0^0=[\textnormal{Id}, \bigotimes_{j=1}^{wf-1}x_j^{r_j}]$ generates a 1-dimensional $KZZ^{\prime}$-submodule isomorphic to $\sigma$, and so does it's image under $T$.
Also, we have
$T^m(s_0^0)=\sum_{s=0}^m\sum_{\mu\in I_s}[g_{s, \mu}^{1-\delta_{s, m}}, 1]$, where $\delta_{s, m}=1$ if $s=m$; and $\delta_{s, m}=0$ if $s\neq m$ (see, equation 7 in \cite{MR3195399}). Thus, $s_n^0+ \beta s_{n-1}^0= T^n(s_0^0)- T^{n-2}(s_0^0)$ for $n\geq 2$.\\
\item For $r=q^w-1$, let us first show that $s_1^0+ \beta s_0^0$ is an eigenvector for the action of cosets of $KZZ^{\prime}/I$. By the identities given in Lemme 2.7 in \cite{MR3195399} we have
\begin{align*}
    \begin{pmatrix} 
1& 0  \\ 
\lambda_0& 1   
 \end{pmatrix}(s_1^0+\beta s_0^0)=\sum_{\mu_0\in I_1}[g_{1, \mu_0}^0, (1-\lambda_0\mu_0)^r&\bigotimes_{j=0}^{wf-1}x_j^{r_j}]+[\alpha, (\lambda_0)^r\bigotimes_{j=0}^{wf-1}y_j^{r_j} ]\\
 &+[g_{1,\lambda_0^{-1}}^0, (-\lambda_0^{-1})^{r}
\bigotimes_{j=0}^{wf-1}x_j^{r_j}].
\end{align*}
Also, we have
\begin{align*}
\begin{pmatrix} 
0& 1  \\ 
1& 0  
 \end{pmatrix}(s_1^0+\beta s_0^0)=&\sum_{\mu_0\in I_1}[g_{1, \mu_0}^0, (-\mu_0)^r\bigotimes_{j=0}^{wf-1}x_j^{r_j}]+[\alpha, \bigotimes_{j=0}^{wf-1}y_j^{r_j} ]\\
 &+[g_{1,0}^0, \bigotimes_{j=0}^{wf-1}x_j^{r_j}].
\end{align*}
Now using the injectivity of $T$ and since we have $T^n(s_1^0+\beta s_0^0)= s_n^0+\beta s_{n-1}^0$, the proof is complete.
\end{enumerate}
\end{proof}
\subsection{Pro-$p$ Iwahori invariants of $\textnormal{ind}_{KZZ^{\prime}}^G\sigma/{(T)}$.}
\begin{prop}\label{invariant}
Let $s_n^k$ and $t_n^s$ be the elements of $\textnormal{ind}_{KZZ^{\prime}}^G\sigma/{(T)}$. Then we have the following.
\begin{enumerate}
    \item If $n\geq 1$ and $k\neq r$, we have $s_n^k\in \textnormal{Im}(T)$ for $0\leq k_j\leq r_j$ and $0\leq j \leq wf-1$.
    \item If $n = 1$, the element $s_1^r$ is a non trivial $I(1)$-invariant. If $n\geq 2$, then $s_n^r\in \textnormal{Im}(T)$.
    \item If $wf>1$, $n\geq 1$ and $0\leq l\leq wf-1$, then the element $s_n^{p^l(r_l+t)}$ is a non trivial $I(1)$-invariant $\mod<\{s_n^{p^l(r_l+s)}\}_{0\leq s\leq t-1}>_G$, where $0\leq t \leq p-r_l-1$.
    \item Set $n\geq 1$. If $e> 1$, $0\leq l\leq wf-1$, and $r> p^s$, then the element $t_n^s$ is a non-trivial $I(1)$-invariant. If $e=1$, then the element $t_n^s$ is non trivial $I(1)$-invariant $\mod<\{s_n^{p^{wf+s-1}m}\}_{1\leq m\leq p-1} >_G$.  
   
\end{enumerate}

\end{prop}
\begin{proof}
\begin{enumerate}
    \item For $n=1$ and $k \neq r$, we set, $v=\bigotimes\limits_{j=0}^{wf-1} x_j^{r_j-k_j}y_j^{k_j} $ and from the equation \eqref{id}, we have
    \begin{equation*}
        T[\textnormal{Id}, v]= \sum_{\lambda_0\in I_1}[g_{1, \lambda_0}^0,  (-\lambda_0)^k\bigotimes\limits_{j=0}^{wf-1} x_j^{r_j}]= (-1)^ks_1^k.
    \end{equation*}
    Now, since $T$ is $G$-equivariant, for $n > 1$, we have 
    \begin{equation*}
        T((-1)^k\sum_{\mu\in I_{n-1}}[g_{n-1,\mu}^0, v])=(-1)^kg_{n-1,\lambda_0}^0T[\textnormal{Id}, v]=g_{n-1,\mu}^0s_1^k= s_n^k,
    \end{equation*}
    which yields $s_n^k \in \textnormal{Im}(T)$.
   \item  Combining the formula (4) in \cite{MR2018825} and Lemma \ref{lem1}, we have the following for $n=1$,
   \begin{equation*}
       T[\textnormal{Id}, \bigotimes_{j=0}^{wf-1}y_j^{r_j}]=\sum_{\lambda_0\in I_1}[g_{1,  \lambda_0}^0,(-\lambda)^r\bigotimes_{j=0}^{wf-1}x_j^{r_j}]+[\alpha, \bigotimes_{j=0}^{wf-1}y_j^{r_j}].
   \end{equation*}
   Since $\beta$ normalizes $I(1)$, we have
   \begin{equation*}
       T[\textnormal{Id}, \bigotimes_{j=0}^{wf-1}y_j^{r_j}]= s_1^r+\beta [\textnormal{Id}, \bigotimes_{j=0}^{wf-1}x_j^{r_j}]
   \end{equation*}
   as an $I(1)$-invariant.
   Now for $n\geq 1$, by the equation \eqref{im(T)0} we have
   \begin{align*}
       T(\sum_{\mu\in I_{n-1}}[g_{n-1, \mu}^0, &\bigotimes_{j=0}^{wf-1}y_j^{r_j}])= \sum_{\mu \in I_{n}}[g_{n, \mu}^0, (-\lambda)^r\bigotimes_{j=0}^{wf-1}x_j^{r_j}]\\
       &+\sum_{\mu\in I_{n-2}}[g_{n-2, [\mu]_{n-1}}^0, \sum_{\mu_{n-2}\in I_1}\bigotimes_{j=0}^{wf-1}(\mu_{n-2}^{p_j}x_j+y_j)^{r_j} ].
   \end{align*}
By Lemma \ref{modp} we have
\begin{equation*}
    \sum_{\mu\in I_{n-2}}[g_{n-2, [\mu]_{n-1}}^0, \sum_{\mu_{n-2}\in I_1}\bigotimes_{j=0}^{wf-1}(\mu_{n-2}^{p_j}x_j+y_j)^{r_j} ]=0.
\end{equation*}
Therefore, we have $\sum_{\mu \in I_{n}}[g_{n, \mu}^0, (-\lambda)^r\bigotimes_{j=0}^{wf-1}x_j^{r_j}]=s_n^r\in \textnormal{Im} (T)$.
   
 \item Let us first consider the case $n=1$ and $0\leq t\leq p-r_l-1$. We have
 \begin{equation*}\begin{pmatrix} 
1& b  \\ 
0& 1   
 \end{pmatrix}s_1^{p^l(r^l+t)}-s_1^{p^l(r_l+t)}= \begin{pmatrix} 
1& b  \\ 
0& 1   
 \end{pmatrix}\sum_{\mu\in I_1}[g_{1, \mu}^0, (\mu_0)^{p^l(r_l+t)}\bigotimes_{j=0}^{wf-1}x_j^{r_j}]-s_1^{p^l(r_l+t)}.\\
\end{equation*}
By Lemma \ref{witt}, we have
\begin{equation}\label{id1}
 \begin{pmatrix} 
1& b  \\ 
0& 1   
 \end{pmatrix}\begin{pmatrix} 
\varpi_D & \mu  \\ 
0& 1   
 \end{pmatrix} =\begin{pmatrix} 
\varpi_D & [\mu_0+b_0]  \\ 
0& 1   
 \end{pmatrix}  \begin{pmatrix} 
1& B(\mu_0,b)  \\ 
0& 1   
 \end{pmatrix},
\end{equation}
 where 
 \begin{equation*}
     B(\mu_0, b)=\varpi_D^{e-1}[P_0(\mu_0, b_0)]+[b_1]+[b_2]\varpi_D]+\cdots.
 \end{equation*}
 By making the substitution $\mu_0\rightarrow \mu_0-b_0$, we get
 \begin{align*}
  \begin{pmatrix} 
1& b  \\ 
0& 1   
 \end{pmatrix}s_1^{p^l(r^l+t)}-&s_1^{p^l(r_l+t)}=\sum_{\mu\in I_1}[g_{1, \mu}^0,  ((\mu_0-b_0)^{p^l(r_l+t)}-(\mu_0)^{p^l(r_l+t)})\bigotimes_{j=0}^{wf-1}x_j^{r_j}] \\
 &=\sum_{\mu\in I_1}[g_{1, \mu}^0, \sum_{s_l=0}^{r_l+t}(-b_0)^{p_l(r_l+t-s_l)}(\mu_0)^{p^ls_l}\binom{r_l+t}{s_l}\bigotimes_{j=0}^{wf-1}x_j^{r_j}]\\
 &=\sum_{s_l=0}^{r_l+t}(-b_0)^{p_l(r_l+t-s_l)}\binom{r_l+t}{s_l}s_1^{p_ls_l}.
 \end{align*}
 It then follows that
 \begin{equation*}
\begin{pmatrix} 
1& b  \\ 
0& 1   
\end{pmatrix}s_1^{p^l(r^l+t)}= s_1^{p^l(r_l+t)}\mod<\{s_n^{p^l(r_l+s)}\}_{0\leq s\leq t-1}>_G.
 \end{equation*}
 Note the invariant of $s_1^{p^l(r_l+t)}$   under 
 \begin{equation*}
     \begin{pmatrix} 
1& 0  \\ 
\varpi_Dc& 1
 \end{pmatrix}\hspace{5pt}\text{and}\hspace{5pt} \begin{pmatrix} 
1+\varpi_Da& 0  \\ 
0& 1
 \end{pmatrix}
 \end{equation*}
 is trivial. The statement for $n > 1$ follows from Corollary 2.6 in \cite{MR3195399} by identifying $X=s_1^{p^l(r_l+t)}$.  
 \item When $n=1$ by the equation \eqref{id1} and making the substitution $\mu_0 \rightarrow \mu_0-b_0$ we get
 \begin{align}
     \begin{pmatrix} 
1& b  \\ 
0& 1   
 \end{pmatrix}t_1^s-t_1^s&=\sum_{\mu\in I_1}[g_{1, \mu}^0, \begin{pmatrix} 
1& B(\mu_0-b_0, b)  \\ 
0& 1   
 \end{pmatrix}\bigotimes\limits_{s\neq j=0}^{wf-1}x_j^{r_j}\otimes x_s^{r_s-1}y_s]-t_1^s\nonumber\\
 &=\sum_{\mu\in I_1}[g_{1, \mu}^0, \bigotimes\limits_{s\neq j=0}^{wf-1}x_j^{r_j}\otimes x_s^{r_s-1}(B(\mu_0-b_0,  b)^{p^s}x_s+y_s)]-t_1^s.\label{rhs of 4}
\end{align}
 Now, we recall that $B(\mu_0, b_0)=\varpi_D^{e-1}[P_0(\mu_0, b_0)]+[b_1]+[b_2]\varpi_D+\cdots$. Thus, using the Kronecker delta function $\delta_{e,1}$, we can write \eqref{rhs of 4} as
 \begin{align*}
      R.H.S&=\sum_{\mu\in I_1}[g_{1, \mu}^0, \bigotimes\limits_{s\neq j=0}^{wf-1}x_j^{r_j}\otimes x_s^{r_s-1}(\delta_{e, 1}P_0(\mu_0-b_0, b_0)+b_1)^{p^s}x_s+y_s)]\\
 &=\sum_{\mu\in I_1}[g_{1, \mu}^0, (\delta_{e, 1}P_0(\mu_0-b_0, b_0))^{p^s} \bigotimes_{j=0}^{wf-1}x_j^{r_j}]+[g_{1, \mu}^0, (b_1)^{p^s}\bigotimes_{j=0}^{wf-1}x_j^{r_j}].\\
 \end{align*}
 When $e=1$, by the equation \eqref{P_0} we note that the degree of $\mu_0$ appearing in $P_0(\mu_0-b_0, b)$ are $\{p^{wf+s-1}\}_{0\leq m\leq p-1}$.\\
Similarly, we have
\begin{align*}
     \begin{pmatrix} 
1& 0  \\ 
\varpi_Dc& 1   
 \end{pmatrix}t_1^s-t_1^s&=\begin{pmatrix} 
1& 0  \\ 
\varpi_Dc& 1   
 \end{pmatrix}\sum_{\mu\in I_1}[g_{1, \mu}^0,\bigotimes\limits_{s\neq j=0}^{wf-1}x_j^{r_j}\otimes x_s^{r_s-1}y_s]-t_1^s.
\end{align*}
Note that
\begin{equation*}
     \begin{pmatrix} 
1& 0  \\ 
\varpi_Dc& 1   
 \end{pmatrix}
 \begin{pmatrix} 
\varpi_D & \mu  \\ 
0 & 1   
 \end{pmatrix}= \begin{pmatrix} 
\varpi_D& \mu  \\ 
0 & 1   
 \end{pmatrix} \begin{pmatrix} 
1-\varpi_D^{-1}\mu\varpi_Dc\varpi_D& -\varpi_D^{-1}\mu\varpi_Dc\mu  \\ 
\varpi_Dc\varpi_D & \varpi_Dc\mu+1   
 \end{pmatrix}
\end{equation*}
 and the very right matrix acts trivially on $\bigotimes_{j=0}^{wf-1}x_j^{r_j}$. It thus follows that
 \begin{align*}
   \begin{pmatrix} 
1& 0  \\ 
\varpi_Dc& 1   
 \end{pmatrix}t_1^s-t_1^s&=  0.\\
 \end{align*}
 Finally, we have
 \begin{equation*}
\begin{pmatrix} 
\varpi_Da+1 & 0  \\ 
0& 1   
\end{pmatrix}t_1^s-t_1^s=\begin{pmatrix} 
\varpi_Da+1 & 0  \\ 
0& 1   
 \end{pmatrix}\sum_{\mu\in I_1}[g_{1, \mu}^0,\bigotimes\limits_{s\neq j=0}^{wf-1}x_j^{r_j}\otimes x_s^{r_s-1}y_s]-t_1^s.
 \end{equation*}
 Note that
 \begin{equation*}
        \begin{pmatrix} 
\varpi_Da+1 & 0  \\ 
0& 1   
 \end{pmatrix}
 \begin{pmatrix} 
\varpi_D & \mu  \\ 
0 & 1   
 \end{pmatrix}= \begin{pmatrix} 
\varpi_D& \mu  \\ 
0 & 1   
 \end{pmatrix} \begin{pmatrix} 
a\varpi_D+1& a\mu  \\ 
0 & 1   
 \end{pmatrix}. 
 \end{equation*}
 Therefore, we have
 \begin{align*}
  \begin{pmatrix} 
\varpi_Da+1 & 0  \\ 
0& 1   
\end{pmatrix}t_1^s-t_1^s=  \sum_{\mu\in I_1}[g_{1, \mu}^0, (a\mu)^{p^s}\bigotimes\limits_{j=0}^{wf-1}x_j^{r_j}]= a^{p^s}s_1^{p^s}. 
 \end{align*}
 
 With all the cases above, we see that if $r> p^s$ and $e > 1$, the element $t_1^s$ is an $I(1)$-invariant. If $e=1$, then the element $t_1^s$ is non trivial $I(1)$-invariant $\mod<\{s_n^{p^{wf+s-1}m}\}_{1\leq m\leq p-1} >_G$. This completes the proof.
\end{enumerate}
\end{proof}
\begin{lem}\label{I}
Let $\begin{pmatrix} 
a & b  \\ 
c& d   
\end{pmatrix}\in I$ be given. Let the elements $s_n^k$ and $t_n^s$ be $I(1)$-invariants as in Proposition \ref{invariant}. Then they are $I$-eigenvectors and those actions are given by,
\begin{enumerate}
    \item $\begin{pmatrix} 
a & b  \\ 
c& d   
\end{pmatrix} s_n^k= a^{k+1}d^{-k} s_n^k$ ; and
\item $\begin{pmatrix} 
a & b  \\ 
c& d   
\end{pmatrix} t_n^s= a^{1-p^s}d^{p^s}t_n^s$.
\end{enumerate}
\end{lem}
\begin{proof}
We have that
\begin{equation*}
    I/I(1)=\Big\{\begin{pmatrix} 
a & 0  \\ 
0& d   
\end{pmatrix} |\hspace{5pt} a, d\in \mathbb{F}_{q^w}^{\times}  \Big\}.
\end{equation*}
It then follows that any elements in $I$ and 
$\begin{pmatrix} 
a & 0  \\ 
0& d   
\end{pmatrix}$
have the same action on $I(1)$-invariants. Thus, we have
\begin{align*}
    \begin{pmatrix} 
a & 0  \\ 
0& d   
\end{pmatrix}s_n^k &= a\begin{pmatrix} 
1 & 0  \\ 
0& a^{-1}d   
\end{pmatrix}\sum_{\mu\in I_n}[g_{n,\mu}^0,  \mu_{n-1}^k\bigotimes_{j=0}^{wf-1}x_j^{r^j}]\\
&=a\sum_{\mu\in I_n}[\begin{pmatrix} 
\varpi_D^n & \mu da^{-1}  \\ 
0& 1  
\end{pmatrix}, \begin{pmatrix} 
1 & 0  \\ 
0& ad^{-1}   
\end{pmatrix}\mu_{n-1}^k\bigotimes_{j=0}^{wf-1}x_j^{r^j}]\\
&=a^{k+1}d^{-k} s_n^k.
\end{align*}
Similarly, we get the $I$-action on the element $t_n^s$ as follows,
\begin{align*}
  \begin{pmatrix} 
a & 0  \\ 
0& d   
\end{pmatrix}t_n^s &= a\begin{pmatrix} 
1 & 0  \\ 
0& a^{-1}d   
\end{pmatrix}\sum\limits_{\mu\in I_n}[g_{n, \mu}^0,  \bigotimes\limits_{s\neq j=0}^{wf-1}x_j^{r_j}\otimes x_s^{r_s-1}y_s]\\  
&=a^{1-p^s}d^{p^s}t_n^s.
\end{align*}
This completes the proof.
\end{proof}

\section{Main Theorem} 
In this section, we provide a full description of the standard basis of $I(1)$-invariants $(\textnormal{ind}_{KZZ^{\prime}}^G\sigma/(T))^{I(1)}$, as stated in Theorem \ref{main}. We shall prove the main theorem by means of several Lemmas.

\subsection{Three supporting Lemmas.}
\begin{lem}\label{lem7}
Let $n\in\mathbb{Z}_{\geq 0}$ be given and $\widetilde{f}$ be an element of $\textnormal{ind}_{KZZ^{\prime}}^G\sigma$, such that $g\widetilde{f}-\widetilde{f}\in T(\textnormal{ind}_{KZZ^{\prime}}^G\sigma)+B_{n-1} \forall g\in I(1)$. For $f^{\prime}\in B_{n-1}$, we set $\widetilde{f}=\widetilde{f}_n+f^{\prime}$, where $\widetilde{f}_n\in X_n$. Further subdivide $\widetilde{f}_n=\widetilde{f}_n^0+\widetilde{f}_n^1$, where $\widetilde{f}_n^i\in X_n^i$. Then we have

\begin{align*}
    \widetilde{f}_n^0= \sum_{\mu\in I_n}[g_{n,\mu}^0, c(\mu)\bigotimes_{j=0}^{wf-1}x_j^{r_j} +\sum\limits_{k=0}^{wf-1}d_k(\mu)\bigotimes_{k\neq j=0}^{wf-1}x_j^{r_j}\otimes x_k^{r_k-1}y_k],\\
    \widetilde{f}_n^1= \sum_{\mu\in I_n}[g_{n,\mu}^1, c^{\prime}(\mu)\bigotimes_{j=0}^{wf-1}y_j^{r_j} +\sum\limits_{k=0}^{wf-1}d^{\prime}_k(\mu)\bigotimes_{k\neq j=0}^{wf-1}y_j^{r_j}\otimes y_k^{r_k-1}x_k],\\
\end{align*}
where $c(\mu), c^\prime(\mu), d(\mu), d^\prime(\mu)$ are identified at $\mu$ of polynomials  
$c, c^\prime, d$ and $d^\prime$ in $\overline{\mathbb{F}}_p[z_0,\cdots,z_{n-1}]$ of degree not greater than $q^w-1$ in each variable $z_j$. 

\end{lem}
\begin{proof}

We are largely utilizing the methods presented in Lemma 3.11 of \cite{MR3873949}. Notice that $\beta^{-1}$ normalizes $I(1)$. It suffices to show the equality for $\widetilde{f}_n^0$ and $n\geq0$ since the equality for $\widetilde{f}_n^1$ can be deduced by applying $\beta^{-1}\widetilde{f}_n^1$. Set $\widetilde{f}_n^0=\sum\limits_{\mu \in I_n}[g_{n,\mu}^0, v_{\mu}]$, where $v_{\mu}=\sum\limits_{\vec{i}=0}^{\vec{r} }c_{\vec{i}}\bigotimes\limits_{j=0}^{wf-1} x_j^{r_j-i_j}y_j^{i_j}$. We consider the following,
\begin{align*}
\begin{pmatrix} 
1& \varpi_D^n  \\ 
0& 1   
 \end{pmatrix}
 \widetilde{f}_n^0-\widetilde{f}_n^0  
& = \sum\limits_{\mu \in I_n}[\begin{pmatrix} 1& \varpi_D^n  \\ 0& 1   \end{pmatrix}g_{n,\mu}^0, v_{\mu}]-\sum\limits_{\mu \in I_n}[g_{n,0}^0, v_{\mu}]. \nonumber
\end{align*}
Note that
\begin{equation*}
  \begin{pmatrix} 1& \varpi_D^n  \\ 0& 1   \end{pmatrix}\begin{pmatrix} \varpi_D^n & \mu  \\ 0& 1   \end{pmatrix}= \begin{pmatrix} \varpi_D^n & \mu  \\ 0& 1   \end{pmatrix} \begin{pmatrix} 1& 1  \\ 0& 1   \end{pmatrix}.
\end{equation*}
Therefore, we have
\begin{align*}
 \begin{pmatrix} 
1& \varpi_D^n  \\ 
0& 1   
 \end{pmatrix}
 \widetilde{f}_n^0-\widetilde{f}_n^0 & = \sum\limits_{\mu \in I_n}[\begin{pmatrix} \varpi_D^n & \mu  \\ 0& 1   \end{pmatrix}, \begin{pmatrix} 1& 1  \\ 0& 1   \end{pmatrix}v_{\mu}]-\sum\limits_{\mu \in I_n}[g_{n,0}^0,v_{\mu}]\nonumber\\
& = \sum\limits_{\mu \in I_n}[g_{n, \mu}^0, \begin{pmatrix} 1& 1  \\ 0& 1   \end{pmatrix}v_{\mu}-v_{\mu}].\nonumber 
\end{align*}
It follows from \eqref{im(T)0} that
\begin{equation}\label{Amu}
\begin{pmatrix} 1& 1  \\ 0& 1   \end{pmatrix}v_{\mu}-v_{\mu}= \sum\limits_{\vec{i}=0}^{\vec{r} }c_{\vec{i}}
\bigotimes\limits_{j=0}^{wf-1} x_j^{r_j-i_j}(x_j+y_j)^{i_j}-v_{\mu}
\in \overline{\mathbb{F}}_p\bigotimes\limits_{j=0}^{wf-1} x_j^{r_j}.
\end{equation}
For the rest of the proof we use the model of a directed graph
following Lemma 3.11 in \cite{MR3873949}. Let $\{e_k\}_{k=0}^{wf-1}$ be the standard basis of $\mathbb{R}^{wf}$ and $X=(X_e, X_v)$ be the directed graph where
\begin{align*}
   X_e & =\{\sum\limits_{k=0}^{wf-1}a_ke_k : 0\leq a_k\leq r_k\}\cap \mathbb{Z}^{wf},\\
   X_v & = \{(\vec{i}, \vec{j}); \vec{i}=\vec{j}+e_k\hspace{4pt} \text{for}\hspace{4pt} \text{some}\hspace{4pt} k, 0\leq k\leq wf-1\}.
\end{align*}
 With the above model, we observe that every $v\in Sym^{\vec{r}}\overline{\mathbb{F}}^2_p$ are functions on $X_e$ taking values in $\overline{\mathbb{F}}_p$, where $\vec{i}$ corresponding to $\bigotimes\limits_{j=0}^{wf-1} x_j^{i_j}y_j^{i_j-r_j}$. In this interpretation, functions corresponding to $\overline{\mathbb{F}}_p\bigotimes\limits_{j=0}^{wf-1}x_j^{r_j}$ are supported on $\{\vec{r}\}$.
 
For a given element $\mathcal{V} \in \eqref{Amu}$, nonzero values of $v_{\mu}$ on a vertex $\vec{i}$ can only contribute to values of $\mathcal{V}$ on another vertex ${\vec{j}}$ if there is a non trivial directed path from $\vec{i}$ to ${\vec{j}}$ satisfying $j_k\leq i_k$ for all $k$ and $j_{k_0}< i_{k_0}$ for some $k_0$. Moreover, for ${\vec{i}}+e_k={\vec{r}}$, the value of $\mathcal{V}$ on ${\vec{i}}$ depends only on the value of $v_{\mu}$ on ${\vec{r}}$. Thus, $v_{\mu}$ can only take nonzero values on $\{\vec{0}\}\cup \{e_k\}_{k=0}^{wf-1}$, and so $c_0(\mu)$ and $\{c_{e_k}(\mu)\}_{k=0}^{wf-1}$ are functions from $I_n$ to $\overline{\mathbb{F}}_p$. By Lemma \ref{poly}, these functions can be represented by polynomials of degree at most $q^w-1$ for all $\mu \in I_n$.

\end{proof}
\begin{lem}\label{lem9}
Let $n\in\mathbb{Z}_{\geq 0}$ and $\widetilde{f}$ be an element of $\textnormal{ind}_{KZZ^{\prime}}^G\sigma$ such that $g\widetilde{f}-\widetilde{f}\in T(\textnormal{ind}_{KZZ^{\prime}}^G\sigma)+B_{n-1} \forall g\in I(1)$. For $f^{\prime}\in B_{n-1}$, we set $\widetilde{f}=\widetilde{f}_n^0+\widetilde{f}_n^1+f^{\prime}$ where support of $\widetilde{f}_n^0$ is $X_n^0$ and support of $\widetilde{f}_n^1$ is $X_n^1$. Then $d_k$ and $d^{\prime}_k$ are constant for all $0\leq k\leq wf-1.$ If $e=1$, then $d_k=d^{\prime}_k=0.$
\end{lem}
\begin{proof}
We follow the methods given in Lemma 3.13 in \cite{MR3873949}.
We begin by taking the case where $e=1$ and $d_k\neq 0$. Applying $\beta^{-1}\widetilde{f}_n^1$, we can deduce the equality for $\widetilde{f}_n^1$. Consider the following equality when $b\in I_1$,
\begin{align*}
\begin{pmatrix} 
1& [b]\varpi_D^{n-1}  \\ 
0& 1   
 \end{pmatrix}\widetilde{f}_n^0-\widetilde{f}_n^0&=\sum\limits_{\mu\in In}[\begin{pmatrix} 
1& [b]\varpi_D^{n-1} \\ 
0& 1   
 \end{pmatrix}g_{n,\mu}^0,c(\mu)\bigotimes\limits_{j=0}^{wf-1}x_j^{r_j}]\\&+[g_{n,\mu}^0,\sum\limits_{k=0}^{wf-1}d_k\bigotimes\limits_{k\neq j=0}^{wf-1}x_j^{r_j}\otimes x_k^{r_k-1}y_k]-\widetilde{f}_n^0
 \end{align*}
 where
  \begin{equation*}
    \widetilde{f}_n^0=\sum\limits_{\mu\in In}[g_{n,\mu}^0, c(\mu)\bigotimes\limits_{j=0}^{wf-1}x_j^{r_j}+\sum\limits_{k=0}^{wf-1}d_k(\mu)\bigotimes\limits_{k\neq j=0}^{wf-1}x_j^{r_j}\otimes x_k^{r_k-1}y_k].
\end{equation*}
 Observe that
 \begin{equation*}
   \begin{pmatrix} 
1& [b]\varpi_D^{n-1} \\ 
0& 1   
 \end{pmatrix}\begin{pmatrix} 
\varpi_D^{n}& \mu \\ 
0& 1   
 \end{pmatrix}  =\begin{pmatrix} \varpi_D^{n}&[\mu]_{n-1}+[\mu_{n-1}+b]\varpi_D^{n-1}  \\ 0& 1   \end{pmatrix}\begin{pmatrix} 1& P_0(\mu_{n-1}, b)  \\ 0& 1   \end{pmatrix},
 \end{equation*}
 where $P_0(\mu_{n-1}, b)$ is the quotient given in Lemma \ref{witt}.
We continue by making the substitution $\mu_{n-1}\rightarrow \mu_{n-1}-b$. Thus we get
 \begin{align}
 \begin{pmatrix} 1& [b]\varpi_D^{n-1}  \\ 0& 1   \end{pmatrix}\widetilde{f}_n^0&-\widetilde{f}_n^0 =\sum\limits_{\mu\in In}[g_{n,\mu}^0,c([\mu]_{n-1},[\mu_{n-1}-b])\bigotimes\limits_{j=0}^{wf-1}x_j^{r_j}]\nonumber
 \\&+[g_{n,\mu}^0,\begin{pmatrix} 1& P_0(\mu_{n-1}-b, b)  \\ 0& 1   \end{pmatrix}\sum\limits_{k=0}^{wf-1}d_k\bigotimes\limits_{k\neq j=0}^{wf-1}x_j^{r_j}\otimes x_k^{r_k-1}y_k)]-\widetilde{f}_n^0\nonumber\\ 
 &=\sum\limits_{\mu\in In}[g_{n,\mu}^0,(c([\mu]_{n-1},[\mu_{n-1}-b])-c([\mu]_{n-1},[\mu_{n-1}]))\bigotimes\limits_{j=0}^{wf-1}x_j^{r_j}]\label{monomial}
 \\&+[g_{n,\mu}^0, \sum\limits_{k=0}^{wf-1}d_k(P_0(\mu_{n-1}-b, b)^{p^k})\bigotimes\limits_{j=0}^{wf-1}x_j^{r_j}].\nonumber
 \end{align}
For $0\leq k \leq wf-1$, Theorem \ref{lucus} implies that
\begin{align*}
    P_0(\mu_{n-1}-b, b)^{p^k}&=(\frac{(\mu_{n-1}-b)^{q^w}+b^{q^w}-\mu_{n-1}^{q^w}}{p})^{p^k}\\
    &=(\sum\limits_{s=1}^{p-1}\mu_{n-1}^{sp^{wf-1}}(-b)^{p^{fw}-sp^{wf-1}}\frac{\binom{p^{fw}}{sp^{fw-1}}}{p})^{p^k}.
\end{align*}
Observe that $P_0(b, \mu_{n-1}-b)^{p^k}$ has a nonzero monomial of the form $(-b)^{p^{k-1}}\mu_{n-1}^{(p-1)p^{k-1}}$ (if $k=0$, it has a monomial of the form $(-b)^{p^{wf-1}}\mu_{n-1}^{(p-1)p^{wf-1}}$).
We assume that $c_{[\mu]_n-1}(x)=\sum\limits_{j=0}^{q^w-1}a_j([\mu_{n-1}])z^j$  
where $a_j\in\overline{\mathbb{F}}_p[x_0,\dots,x_{n-1}]$ are suitable polynomial which exists by Lemma \ref{poly}. Let us set polynomials, 
$c([\mu]_{n-1}, [\mu_{n-1}])=c_{[\mu]_n-1}(\mu_{n-1})$ and $c([\mu]_{n-1}, [\mu_{n-1}-b])=c_{[\mu]_{n-1}}(\mu_{n-1}-b)$. We write,
\begin{align}
    \sum_{j=0}^{q^w-1}\widetilde{a}_j([\mu_{n-1}])\mu_{n-1}^j=&c_{[\mu]_{n-1}}(\mu_{n-1}-b)-c_{[\mu]_{n-1}}(\mu_{n-1})\nonumber\\
    =&\sum_{j=0}^{q^w-1}a_j([\mu_{n-1}])(\sum_{s=0}^{j-1}(-b)^{j-s}\binom{j}{s}\mu_{n-1}^s).\label{expand}
\end{align}
Since \eqref{monomial} lies in $\textnormal{Im}(T)$ and $r_k<p-1$, we have
\begin{equation*}
    \mu^{(p-1)p^k}(\sum_{j=0}^{q^w-1}\widetilde{a}_{(p-1)p^{k-1}}([\mu_{n-1}])+(-b)^{p^{k-1}})=0.
\end{equation*}
This give us, $\widetilde{a}_{(p-1)p^{k-1}}([\mu_{n-1}])+(-b)^{p^{k-1}}=0$ for all $b\in \mathbb{F}_{q^w}^\times$. Expanding \eqref{expand}, for each $b\in\mathbb{F}_{q^w}^\times$ we have an equation,
\begin{equation*}
    \widetilde{a}_(p-1)p^{k-1}([\mu_{n-1}])=\sum_{l=(p-1)p^{k-1}+1}^{q^w-1}a_l([\mu_{n-1}])(-b)^{l-(p-1)p^{k-1}}\binom{l}{(p-1)p^{k-1}}.
\end{equation*}
Now it follows from Theorem \ref{lucus} that
\begin{equation*}
   \text{L.H.S}=\sum_{l=(p-1)p^{k-1}+1}^{q^w-1}a_l([\mu_{n-1}])(-b)^{l-(p-1)p^{k-1}}\binom{l_{k-1}}{p-1}=-(-b)^{p^{k-1}}.
\end{equation*}
Multiplying each of these equations by $(-b)^{-p^{k-1}}$, and summing them together,
\begin{equation}\label{controdiction}
    \sum_{l=(p-1)p^{k-1}+1}^{q^w-1}a_l([\mu_{n-1}])(-b)^{l-p^k}\binom{l_{k-1}}{p-1}=\sum_{b\in\mathbb{F}_{q^w}^{\times}}-1.
\end{equation}

By Lemma \ref{modp}, an element in the left hand side of \eqref{controdiction} is not equal to zero when $l-p^k\equiv 0\hspace{4pt}\text{mod} \hspace{4pt}q^w-1$. Hence, if $l\equiv p^k\hspace{4pt}\text{mod}\hspace{4pt}q^w-1$, we must have $l_{k-1}=0$. Applying Theorem \ref{lucus}, this yields $\binom{l_{k-1}}{p-1}=0$. This leads to a contradiction, and we conclude $d_k=0$ for $0\leq k \leq wf-1$.\\
For $e \neq 1$, consider \begin{equation*}
    \widetilde{f}_n^0=\sum\limits_{\mu\in In}[g_{n,\mu}^0, c(\mu)\bigotimes\limits_{j=0}^{wf-1}x_j^{r_j}+\sum\limits_{k=0}^{wf-1}d_k(\mu)\bigotimes\limits_{k\neq j=0}^{wf-1}x_j^{r_j}\otimes x_k^{r_k-1}y_k].
\end{equation*}

We shall prove that $d_k$ and $d_k^{\prime}$ are constants using mathematical induction. Let us assume that $d_k$ is independent of $\mu_{n-i}$ for $i \leq m$. If $\mu\in I_n$, $\mu_n \notin \mu$ and thus the claim holds true when $m=0$. Assuming $d_k$ is independent of $\mu_{n-1}, \cdots, \mu_{n-m+1}$ and considering the following equality, 
\begin{align*}
\begin{pmatrix} 1& \varpi_D^{n-m}  \\ 0& 1   \end{pmatrix}\widetilde{f}_n^0-\widetilde{f}_n^0&= \sum\limits_{\mu\in In}[\begin{pmatrix} 1& \varpi_D^{n-m}  \\ 0& 1   \end{pmatrix}g_{n,\mu}^0,c(\mu)\bigotimes\limits_{j=0}^{wf-1}x_j^{r_j}]\\&+[\begin{pmatrix} 1& \varpi_D^{n-m}  \\ 0& 1   \end{pmatrix}g_{n,\mu}^0,\sum\limits_{k=0}^{wf-1}d_k(\mu)\bigotimes\limits_{k\neq j=0}^{wf-1}x_j^{r_j}\otimes x_k^{r_k-1}y_k]-\widetilde{f}_n^0.
\end{align*}
Observe that
\begin{equation*}
    \begin{pmatrix} 1& \varpi_D^{n-m}  \\ 0& 1   \end{pmatrix}\begin{pmatrix} \varpi_D^n & \mu  \\ 0& 1   \end{pmatrix}=\begin{pmatrix} \varpi_D^n & \mu^\prime  \\ 0& 1   \end{pmatrix} \begin{pmatrix} 1 & z  \\ 0& 1   \end{pmatrix},
\end{equation*}
where $\mu_t^\prime=\mu_t$ for $0\leq t <n-m$, $\mu_{n-m}+1=\mu^\prime_{n-m}$ and $z \in\mathcal{O}_D$. Also, we have $\mu_t^\prime\neq \mu_t$ for $t>n-m$ but since $d_k$ is independent of $\mu_{n-1}, \cdots, \mu_{n-m+1}$ the transformation $\mu^{\prime}\mapsto \mu $ does not effect $d_k(\mu)$. Thus, we have  
\begin{align*}
   \begin{pmatrix} 1& \varpi_D^{n-m}  \\ 0& 1   \end{pmatrix}&\widetilde{f}_n^0-\widetilde{f}_n^0=\sum\limits_{\mu\in I_n}[g_{n,\mu}, c^{\prime}(\mu)\bigotimes\limits_{j=0}^{wf-1} x_j^{r_j}]\\+[g_{n, \mu}^0,\sum\limits_{k=0}^{wf-1}&d_k([\mu]_{n-m},\mu_{n-m}-1) \bigotimes\limits_{k\neq j= 0}^{wf-1}x_j^{r_j}\otimes z^{p^k}x_k^{r_k}+x_k^{r_k-1}y_k)]-\widetilde{f}_n^0,\\
   \end{align*}
   where $c^{\prime}(\mu)$ is a polynomial in $\overline{\mathbb{F}}_p[\mu_0, \cdots, \mu_{n-1}]$.
 This gives us,
  \begin{align}
   \begin{pmatrix} 1& \varpi_D^{n-m}  \\ 0& 1   \end{pmatrix}\widetilde{f}_n^0-\widetilde{f}_n^0&=\sum\limits_{\mu\in I_n}[g_{n,\mu}^0,\widetilde{c}(\mu)\bigotimes\limits_{j=0}^{wf-1} x_j^{r_j}]\nonumber\\
   &+[g_{n,\mu}^0, \sum\limits_{k=0}^{wf-1}(\widetilde{d}_k({\mu}) )\bigotimes\limits_{k\neq j= 0}^{wf-1}x_j^{r_j}\otimes x_k^{r_k-1}y_k)] \label{d_k} \\
   &\in\textnormal{Im}(T)\nonumber,
 \end{align}
 where $\widetilde{d}_k({\mu})=d_k([\mu]_{n-m},\mu_{n-m}-1)-d_k([\mu]_{n-m},\mu_{n-m})$ and $\widetilde{c}(\mu)$ is a polynomial in $\overline{\mathbb{F}}_p[\mu_0, \cdots, \mu_{n-1}]$.

Therefore, we have
\begin{equation*} 
d_k([\mu]_{n-m},\mu_{n-m}-1)-d_k([\mu]_{n-m},\mu_{n-m})=0
\end{equation*}
for all $k$ since $\eqref{d_k}\in \textnormal{Im} (T) $ and thus $d_k$ is independent of $\mu_{n-m}$ and by induction, it is a constant. 
\end{proof}

\begin{lem}\label{lem 13}
Let $n\in\mathbb{Z}_{\geq 0}$ and $\widetilde{f}$ be an element of $\textnormal{ind}_{KZZ^{\prime}}^G\sigma$ such that $g\widetilde{f}-\widetilde{f}\in T(\textnormal{ind}_{KG}^G\sigma)+B_{n-1} \forall g\in I(1)$. For $f^{\prime}\in B_{n-1}$, we set $\widetilde{f}=\widetilde{f}_n^0+\widetilde{f}_n^1+f^{\prime}$ where support of $\widetilde{f}_n^0$ is $S_n^0$ and support of $\widetilde{f}_n^1$ is $S_n^1$. Then $c(\mu)=\sum_{k=0}^{wf-1}\widetilde{c}_k\mu_{n-1}^{p^k(r_k+1)}$ and $c^{\prime}(\mu)=\sum_{k=0}^{wf-1}\widetilde{c}_k^{\prime}\mu_{n-1}^{p^k(r_k+1)}$. Furthermore, $\widetilde{c}_k$ and $\widetilde{c}_k^{\prime}$ are independent of $\mu$, where $c(\mu)$ and $c^{\prime}(\mu)$ are as define in lemma \ref{lem7}.
\end{lem}
\begin{proof}
Assume $wf > 1$. By Lemma \ref{lem9} we have
\begin{align*}
    \widetilde{f}_n^0&=\sum\limits_{\mu \in I_n}[g_{n,\mu}^0, {c}(\mu)\bigotimes\limits_{j=0}^{wf-1}x_j^{r_j}]+\sum\limits_{\mu \in I_n}[g_{n,\mu}^0, \bigotimes\limits_{s\neq j=0}^{wf-1}x_j^{r_j}\otimes x_s^{r_s-1}y_k]\\
    &=\sum\limits_{\mu \in I_n}[g_{n,\mu}^0, c(\mu)\bigotimes\limits_{j=0}^{wf-1}x_j^{r_j}]+ t_n^s.
\end{align*}

By Proposition \ref{invariant}, the element $y_n^s$ is an $I(1)$-invariant. Therefore, we can assume without loss of generality that $\widetilde{f}_n^0=\sum\limits_{\mu \in I_n}[g_{n,\mu}^0, {c}(\mu)\bigotimes\limits_{j=0}^{wf-1}x_j^{r_j}]$ holds for all $n$. For the case $n=0$, this claim is trivial, since $\mu \in I_0$.
Now for $n \geq 1$, consider the following computation,
\begin{align*}
    \begin{pmatrix} 
1& -\varpi_D^{n-1}  \\ 
0& 1   
 \end{pmatrix}
 \widetilde{f}_n^0&-\widetilde{f}_n^0 = \begin{pmatrix} 
1& -\varpi_D^{n-1}  \\ 
0& 1   
\end{pmatrix}
\sum\limits_{\mu \in I_n}[g_{n,\mu}^0,  c(\mu)\bigotimes\limits_{j=0}^{wf-1}x_j^{r_j}]-\widetilde{f}_n^0. \\
&=\sum\limits_{\mu \in I_n}[g_{n, \mu}^0, (c([\mu]_{n-1}, [\mu_{n-1}+1])-c([\mu]_{n-1}, [\mu_{n-1}]))\bigotimes\limits_{j=0}^{wf-1}x_j^{r_j}].
\end{align*}
Let us denote 
\begin{equation*}
    \bigtriangleup c= c([\mu]_{n-1}, [\mu_{n-1}+1])-c([\mu]_{n-1}, [\mu_{n-1}]),
\end{equation*}
and it is a polynomial in $\mu_{n-1}$ with coefficients in $\overline{\mathbb{F}}_p[\mu_0,\cdots,\mu_{n-2}]$. 

We claim that the degree $k$ of $\mu_{n-1}$ must be between $0$ and $r_j$ for all $0 \leq j \leq wf-1$, or equal to $p^l(r_l+1)$ for some $0 \leq l \leq wf-1$. If this is not true, then there is a $k_{j_0} > r_{j_0}$ for some $j_0$ and $k \neq p^{j_0}(r_{j_0}+1)$. Without loss of generality, we can assume there is no other monomial $\mu_{n-1}^{k^\prime}$ in $c(\mu)$ such that $k_j \leq k_{j^\prime}$ for all $j$.
By Theorem \ref{lucus} we have
\begin{equation*}
    (\mu_{n-1}+1)^k-\mu_{n-1}^k=\sum\limits_{i=0}^{k-1}\prod\limits_{j=0}^{wf-1}\left(
    \begin{array}{c}
      k_j \\
      i_j
    \end{array}
  \right)\mu_{n-1}^i, 
\end{equation*}
and the above implies that $\bigtriangleup c$ contains all the monomials of the form $\mu_{n-1}^{k-p^l}$ where $0\leq l \leq wf-1$. In particular $\mu_{n-1}^{k-p^{j_1}}$ appears in $\bigtriangleup c$ for $j_1$ such that $j_1\neq j_0$ and $k_{j_1}>0$ and since $k_{j_0}>r_{j_0}$ and $k\neq p^{j_0}(r_{j_0}+1)$. This contradicts our assumption. Therefore we write,
\begin{equation*}
      \widetilde{f}_n^0=\sum\limits_{\mu \in I_n}[g_{n,\mu}^0,  (Q([\mu]_{n-1})\mu_{n-1}^r+\sum\limits_{k=0}^{wf-1}\widetilde{c}_k([\mu]_{n-1})\mu_{n-1}^{p^k(r_k+1)})\bigotimes\limits_{j=0}^{wf-1}x_j^{r_j}], 
\end{equation*}
where $\widetilde{c}_k$ and $Q$ depends only on $[\mu]_{n-1}$.\\
Let us now assume $n\geq 2$. Consider the following computation

\begin{align}
    \sum\limits_{\mu \in I_n}[g_{n,\mu}^0, Q([\mu]_{n-1})&\mu_{n-1}^r\bigotimes\limits_{j=0}^{wf-1}x_j^{r_j}] =\sum\limits_{\mu\in I_n}[\begin{pmatrix} 
\varpi_D^{n}& \mu  \\ 
0& 1   
 \end{pmatrix}, Q([\mu]_{n-1})\mu_{n-1}^r\bigotimes\limits_{j=0}^{wf-1}x_j^{r_j}]\label{f_n^0}\\
    &=
    \sum\limits_{\mu\in I_{n-1}}\begin{pmatrix} 
\varpi_D^{n-1}& \mu  \\ 
0& 1   
 \end{pmatrix}Q(\mu)\sum_{\mu^\prime\in I_1}[\begin{pmatrix} 
\varpi_D& \mu^\prime  \\ 
0& 1   
 \end{pmatrix}, 
\mu_{n-1}^r\bigotimes\limits_{j=0}^{wf-1}x_j^{r_j}]\nonumber\\
 &= \sum\limits_{\mu\in I_{n-1}}g_{{n-1},\mu}^0Q(\mu)s_1^r.\label{Q}
\end{align}
By Proposition \ref{T} we have
\begin{equation*}
    (-1)^rs_1^r+[\alpha, \bigotimes\limits_{j=0}^{wf-1}y_j^{r_j}]\in \textnormal{Im}(T)
\end{equation*}
and after taking modulo $\textnormal{Im}(T)$ we obtain
\begin{align*}
(-1)^rs_1^r+[\alpha, \bigotimes\limits_{j=0}^{wf-1}y_j^{r_j}] & =0\\
(-1)^{r+1}s_1^r &=[\alpha, \bigotimes\limits_{j=0}^{wf-1}y_j^{r_j}].
\end{align*}
Thus the equation \eqref{Q} yields
\begin{align}
    \sum\limits_{\mu\in I_{n-1}}g_{{n-1},\mu}^0Q(\mu)s_1^r &= \sum\limits_{\mu\in I_{n-1}}g_{{n-1},\mu}^0[\alpha, (-1)^{r+1}Q(\mu)\bigotimes\limits_{j=0}^{wf-1}y_j^{r_j}]\nonumber\\
    &= \sum\limits_{\mu\in I_{n-1}}[g_{{n-2}, [\mu]_{n-2}}^0,  (-1)^{r+1}Q(\mu)\begin{pmatrix} 
1& [\mu_{n-2}]  \\ 
0& 1   
 \end{pmatrix}\bigotimes\limits_{j=0}^{wf-1}y_j^{r_j}]\label{f_n^1}.
\end{align}
Thus we can replace every term in left hand side of \eqref{f_n^0} by a term in \eqref{f_n^1} without changing $\widetilde{f}_n^1$. 
If $n=1$, we can just replace $s_1^r$ by $(-1)^{r+1}[\alpha ,\bigotimes\limits_{j=0}^{wf-1}y_j^{r_j}]$. 
Therefore we must have $Q(\mu)=0$.

For the rest of the proof we shall use the mathematical induction. 
We inductively assume that $\widetilde{c}_k$ are independent of $[\mu]_{n-m}$  for   $1\leq m \leq n-1$. We note that $\mu \in I_0$ when $n=1$ and thus claim is trivially true. Therefore let us  assume $n\geq 2$ and 
$\widetilde{c}_k$ are independent of $\mu_{n-i}$  for all $i<m$ and observe,
\begin{equation}\label{c_k}
\begin{pmatrix} 
1& \varpi_D^{n-m}  \\ 
0& 1   
 \end{pmatrix}
 \widetilde{f}_n^0-\widetilde{f}_n^0 = \begin{pmatrix} 
1& \varpi_D^{n-m}  \\ 
0& 1   
\end{pmatrix}
\sum\limits_{\mu \in I_n}[g_{n,\mu}^0,  \widetilde{c}([\mu]_{n-1})\mu_{n-1}^{p^k(r_k+1)}\bigotimes\limits_{j=0}^{wf-1}x_j^{r_j}]-\widetilde{f}_n^0. 
\end{equation}
Note that
\begin{equation*}
    \begin{pmatrix} 1& \varpi_D^{n-m}  \\ 0& 1   \end{pmatrix}\begin{pmatrix} \varpi_D^n & \mu  \\ 0& 1   \end{pmatrix}=\begin{pmatrix} \varpi_D^n & \mu^\prime  \\ 0& 1   \end{pmatrix} \begin{pmatrix} 1 & \widetilde{z}  \\ 0& 1   \end{pmatrix}
\end{equation*}
where $\mu_t^\prime=\mu_t$ for $0\leq t <n-m$, $\mu_{n-m}+1=\mu^\prime_{n-m}$ and $z \in\mathcal{O}_D$. Also, we have $\mu_t^\prime = \mu_t+ z_t$ for $t>n-m$. Here $z_t\in \overline{\mathbb{F}}_p[\mu_{n-m}, \cdots, \mu_{n-2}]$. Thus, the right hand side of the equation \eqref{c_k} equals to, 
\begin{align*}
&\sum\limits_{\mu \in I_n}[g_{n,\mu}^0, (\widetilde{c}([\mu]_{n-m}, \mu_{n-m}-1)-\widetilde{c}([\mu]_{n-m}, \mu_{n-m}))  \mu_{n-1}^{p^k(r_k+1)}\bigotimes\limits_{j=0}^{wf-1}x_j^{r_j}]\nonumber\\
&+ \sum\limits_{\mu \in I_n}[g_{n,\mu}^0, \sum\limits_{k=0}^{wf-1}\sum\limits_{j=0}^{r_k}\widetilde{c}([\mu]_{n-m}, \mu_{n-m}-1)(\mu_{n-1}+z_{n-1})^{p^k(r_k+1)}\bigotimes\limits_{j=0}^{wf-1}x_j^{r_j}]\nonumber\\
\end{align*}
\begin{align}
&=
\sum\limits_{\mu \in I_n}[g_{n,\mu}^0, (\widetilde{c}([\mu]_{n-m}, \mu_{n-m}-1)-\widetilde{c}([\mu]_{n-m}, \mu_{n-m}))\mu_{n-1}^{p^k(r_k+1)}\bigotimes\limits_{j=0}^{wf-1}x_j^{r_j}]\label{degree}\\
&+ \sum\limits_{\mu \in I_n}[g_{n,\mu}^0, \sum\limits_{k=0}^{wf-1}\sum\limits_{j=0}^{r_k}\widetilde{c}([\mu]_{n-m}, \mu_{n-m}-1)\mu_{n-1}^{jp^k}z_{n-1}^{p^k(r_k+1)-jp^k} \binom{r_k+1}{j}
\bigotimes\limits_{j=0}^{wf-1}x_j^{r_j}].\nonumber
\end{align}

We note that the degree of $\mu_{n-1}$ is greater than $p^kr_k$ in \eqref{degree} and thus by Proposition \ref{T} we have
\begin{equation*}
    \widetilde{c}([\mu]_{n-m}, \mu_{n-m}-1)-\widetilde{c}([\mu]_{n-m},\mu_{n-m})=0.
\end{equation*}
Therefore we can conclude that $\widetilde{c}([\mu]_{n-m}, \mu_{n-m}-1)=\widetilde{c}([\mu]_{n-m},\mu_{n-m})$. This implies $\widetilde{c}_k$ independent of $[\mu]_{n-m}$ as required.\\
The same is applied to $\widetilde{f}_n^1$ by applying $\beta^{-1}\widetilde{f}_n^1$ to get the result .
\end{proof}
\subsection{Proof of the main theorem.}
We now prove our main theorem of this paper, which provides a full description of a standard basis of $I(1)$-invariants of $(\textnormal{ind}_{KZZ^{\prime}}^G/(T))$ as an $\overline{\mathbb{F}}_p$-vector space.
\begin{thm}\label{main}
Let $D$ be a finite dimensional central $F$-division algebra with $dim_F(D)=w^2$ and let $e$ and $f$ respectively be the ramification degree and inertia degree of $F$ over $\mathbb{Q}_p$. We assume 
$2<r_j< p-3$ for $0\leq j \leq wf-1$ and let us set,
\begin{align*}
    &S_m^l=\{s_n^{p^l(r_l+1)}\}_{n\geq m}\cup \{\beta s_n^{p^l(r_l+1)}\}_{n\geq m},\\
    &T_m^l=\{t_n^l\}_{n\geq m}\cup \{\beta t_n^l\}_{n\geq m},\\
   &S_m=\cup_{l=0}^{wf-1}S_m^l,\\
    &T_m=\cup_{l=0}^{wf-1}T_m^l.\\
\end{align*}
Then, an $I$-eigenbasis for the space $(\textnormal{ind}_{KZZ^{\prime}}^G\sigma/(T))^{I(1)}$ of $I(1)$-invariants as an $\overline{\mathbb{F}}_p$-vector space is given by
\begin{align*}
  S_1\cup \{[\textnormal{Id}, \bigotimes\limits_{j=0}^{wf-1} x_j^{r_j}], [\alpha,\bigotimes\limits_{j=0}^{wf-1} y_j^{r_j} ]\} &\hspace{10pt} \text{when} \hspace{5pt}e=1,wf>1,\\
   S_1\cup \{[\textnormal{Id},\bigotimes\limits_{j=0}^{wf-1} x_j^{r_j}], [\alpha,\bigotimes\limits_{j=0}^{wf-1} y_j^{r_j} ]\}\cup T_1 &\hspace{10pt} \text{when}\hspace{5pt} e>1,wf>1.\\
\end{align*}
\end{thm}
\begin{proof}
Lemmas \ref{lem7}, \ref{lem9} and \ref{lem 13} imply that $\Tilde{f}^0_n$ and $\Tilde{f}^1_n$ contain linear combinations of elements $s_n^{p^k(r_k+1)}$ and $t_n^s$ (respectively $\beta s_n^{p^k(r_k+1)}$ and $\beta t_n^s$). Proposition \ref{invariant} shows that these elements are $I(1)$-invariant in the quotient. Moreover, the image of $\Tilde{f}-\Tilde{f}^0_n-\Tilde{f}_n^1$ lies in the quotient and is supported on $B_{n-1}$, implying that it is also $I(1)$-invariant. For $n=0$, we get $\Tilde{f}_0=\Tilde{f}-f^\prime=[\textnormal{Id}, v_0]+[\alpha, v_1]$ where $f^\prime$ is a $I(1)$-invariant in the quotient and supported on $B_n \backslash B_0$, implying that $\Tilde{f}_0\in (\textnormal{ind}_{KZZ^{\prime}}^G\sigma/(T))^{I(1)}$. By Proposition \ref{T}, we show that 
\begin{equation*}
\begin{pmatrix} 
1& 1 \\ 
0& 1   
 \end{pmatrix}
 \widetilde{f}_0-\widetilde{f}_0   = 
[\textnormal{Id},\begin{pmatrix} 
1& 1  \\ 
0& 1   
\end{pmatrix} v_0]-[\textnormal{Id}, v_0] \in \textnormal{Im}(T)
\end{equation*}
if
\begin{equation*}
   \begin{pmatrix} 
1& 1 \\ 
0& 1   
 \end{pmatrix}
 v_0-v_0= \sum\limits_{k=0}^{wf-1}\widetilde{c}_k\bigotimes\limits_{k\neq j=0}^{wf-1}x_j^{r_j}\otimes x_k^{r_k-1}(x_k+y_k)-\bigotimes\limits_{k\neq j=0}^{wf-1}x_j^{r_j}\otimes x_k^{r_k-1}y_k=0.
\end{equation*}
The above leads to $\widetilde{c}_k=0$ and $v_0=\bigotimes\limits_{ j=0}^{wf-1}x_j^{r_j}$. Similarly, we consider the following argument for $[\alpha,v_1]$,
\begin{equation*}
\begin{pmatrix} 
1& 0 \\ 
\varpi_D& 1   
 \end{pmatrix}
 \widetilde{f}_0-\widetilde{f}_0   = 
[\alpha,\begin{pmatrix} 
1& 0  \\ 
1& 1   
\end{pmatrix} v_1]-[\alpha, v_1] \in \textnormal{Im}(T).
\end{equation*}
This implies that
\begin{equation*}
   \begin{pmatrix} 
1& 0 \\ 
1& 1   
 \end{pmatrix}
 v_1-v_1= \sum\limits_{k=0}^{wf-1}\widetilde{c}_k^{\prime}\bigotimes\limits_{k\neq j=0}^{wf-1}y_j^{r_j}\otimes y_k^{r_k-1}(x_k+y_k)-\bigotimes\limits_{k\neq j=0}^{wf-1}y_j^{r_j}\otimes y_k^{r_k-1}x_k=0.
\end{equation*}
Thus, we conclude $\widetilde{c}_k^{\prime}=0$ and  $v_1=\bigotimes_{j=1}^{wf-1}y_j^{r_j}$. This completes the proof. 
\end{proof}
\bibliographystyle{amsplain}
\bibliography{reference}
\end{document}